\documentclass{article}
\usepackage{amssymb,amsmath,amsfonts,amsthm,graphicx,ifthen,mathrsfs,hyperref}
\usepackage[mathcal]{euscript}
\usepackage{pgf,tikz}
\usetikzlibrary{arrows}
\usepackage{enumitem}

\usepackage[cmtip,arrow]{xy}
\usepackage{comment}
\usepackage{mathtools}

\usepackage{cleveref}
\usepackage[symbol]{footmisc}

\theoremstyle{plain}
\newtheorem{Theorem}{Theorem}   [section]

\newtheorem{Corollary}[Theorem] {Corollary}
\newtheorem{Definition}[Theorem] {Definition}

\newtheorem{Proposition}[Theorem] {Proposition}

\newtheorem{Ruleset}[Theorem] {Ruleset}

\begin{document}
\title{An investigation into the application of genetic programming to combinatorial game theory}
\maketitle
\begin{center}
  {\sc Melissa A. Huggan}\textsuperscript{1}\footnote[2]{Supported by the Natural Sciences and Engineering Research Council of Canada (funding reference number PDF-532564-2019).},
  {\sc Craig Tennenhouse}\textsuperscript{2}
\end{center}
  \par \bigskip
  \textsuperscript{1}Ryerson University, Toronto, ON, Canada, \href{mailto:Melissa.Huggan@ryerson.ca}{melissa.huggan@ryerson.ca} \par
  \textsuperscript{2}University of New England, Biddeford, ME 04005, USA, \href{mailto:ctennenhouse@une.edu}{ctennenhouse@une.edu} \par

\date{}

\maketitle

\begin{abstract} 
\noindent
Genetic programming is the practice of evolving formulas using crossover and mutation of genes representing functional operations. Motivated by genetic evolution we develop and solve two combinatorial games, and we demonstrate some advantages and pitfalls of using genetic programming to investigate Grundy values. We conclude by investigating a combinatorial game whose ruleset and starting positions are inspired by genetic structures. 
\end{abstract}


\section{Introduction}\label{sec:intro}
The fundamental unit of biological evolution is a gene, which represents a small piece of information, and the genome is a collection of genes that encodes an organism's complete genetic information. Within the context of biological evolution, the genes of the most fit organisms survive and are passed onto the next generation, with their chromosomes modifying over time to better fit their environment through competition. This modification occurs through the processes of mutation and crossover, wherein individual genes are altered and pairs of chromosomes trade information, respectively, as organisms pass down their genetic information to their progeny (see Figure~\ref{fig:cross-mut}).

\begin{figure}[h]
\centering
\includegraphics[width=.6\textwidth]{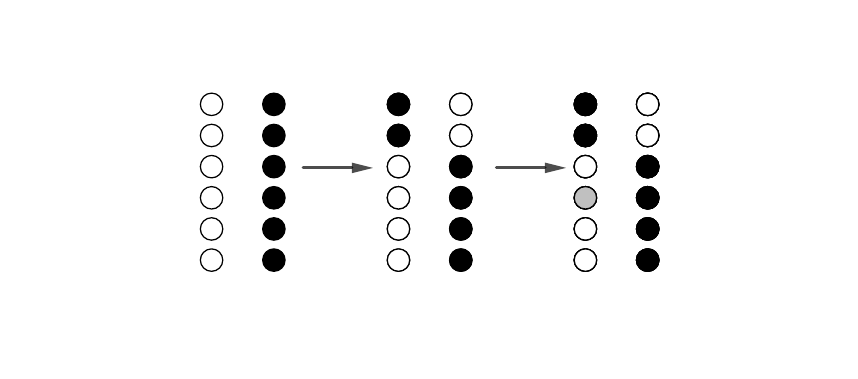}
\caption{A pair of chromosomes undergoing crossover then mutation.}
\label{fig:cross-mut}
\end{figure}

This set of mechanisms in biological evolution has been co-opted as a model for algorithmic development of heuristic solutions to a variety of problems, like antenna design \cite{hornby2006automated}, the Traveling Salesman Problem \cite{brady1985optimization}, and graph coloring \cite{galinier1999hybrid}. In these problems a chromosome encodes information about the structure and properties of working solutions. These solutions are the results of \emph{genetic algorithms}. When the chromosome instead represents a function or program the process is called \emph{genetic programming}. Genetic programming is often used when a user has a collection of data points and is looking for a function to fit them. The fitness of a particular program is therefore related to the error between the data points and the program. This mechanism is similar to that of regression in statistical methods (Figure~\ref{fig:regression}).

\begin{figure}[h]
\centering
\includegraphics[width=.3\textwidth]{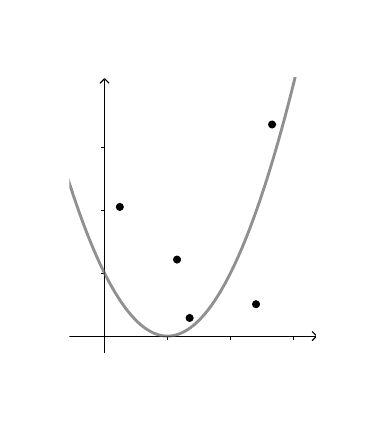}
\includegraphics[width=.3\textwidth]{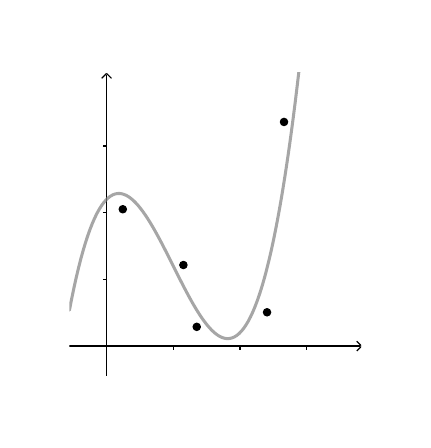}
\caption{A curve with poor fit to data points (left) and another with much better fit (right).}
\label{fig:regression}
\end{figure}

There are a number of different structures used in genetic programming to represent a chromosome, the simplest being a linear structure and a tree structure. A linearly organized chromosome can be visualized much like a biological chromosome (see Figure~\ref{fig:cross-mut}). One example of such a structure is in \cite{oltean2004evolving}, which introduces Multi-Expression Programming. Chromosomes organized into trees have some advantages over the linear approach, and we will discuss them further in Section~\ref{sec:methods}. As chromosomes undergo crossover and mutation, the error tends to decrease and the genetic program evolves to progressively better fit the goal data. 

Genetic algorithms have been applied to combinatorial games, (see~\cite{HauptmanS2005, SAHS2007}). However, these efforts have been focused on using genetic programming to develop strategies rather than finding a formula for the Grundy values. We are interested in examining whether genetic programming could be a useful tool for determining values of combinatorial game positions. The only model for this type of project in the literature is in \cite{oltean2004evolving}, which uses the Multi-Expression Programming model with a linear chromosome. This method does not precisely fit our needs, as the author of that paper focuses on the outcome class classification problem instead of Grundy values, and restricts their investigation to \textsc{nim}. However, the project and its success serve as a strong motivator for the application of genetic programming to combinatorial games, and we hope we have done it justice in extending their results and adding to the body of work combining these two mathematical endeavors. 

Recall that the \emph{Grundy value} of an impartial game position is the smallest non-negative integer not included in the Grundy values of its options \cite{sprague1935mathematische,grundy1939mathematics}. For more information about combinatorial game theory, see~\cite{Albert, Berlekamp, conway2000numbers}. In this project we generate data points of the form $(H,g)$, where $H\in \mathbb{Z}^n$ is a list of integers representing a game position, and $g\in \mathbb{Z}$ is the associated Grundy value. Of the myriad difficulties to overcome in the project, not least of which is the fact that heuristics are not often useful for calculating Grundy values. In truth, either a function completely determines the value of a game or it is incorrect. This leaves us with the difficult task of devising a fitness function that represents distance, not a natural concept in the space of impartial game values, and at the same time leads to eventual convergence with an error of zero.

\subsection{Genetic Programming: Methods}\label{sec:methods}

Since we are interested in data points that are computationally inexpensive to determine, we have chosen to use the Python package \texttt{gpLearn} \cite{stephens2015gplearn}. This package uses the tree model of chromosome representation, introduced in the Introduction to Section~\ref{sec:intro}, wherein each leaf is associated with a primitive (a constant or a single input parameter), and each internal node a function on its child node(s). The root node is therefore recursively associated with a single function on the set of primitives (see Figure~\ref{fig:gp-tree}).

\begin{figure}[h]
\centering
\includegraphics[width=.4\textwidth]{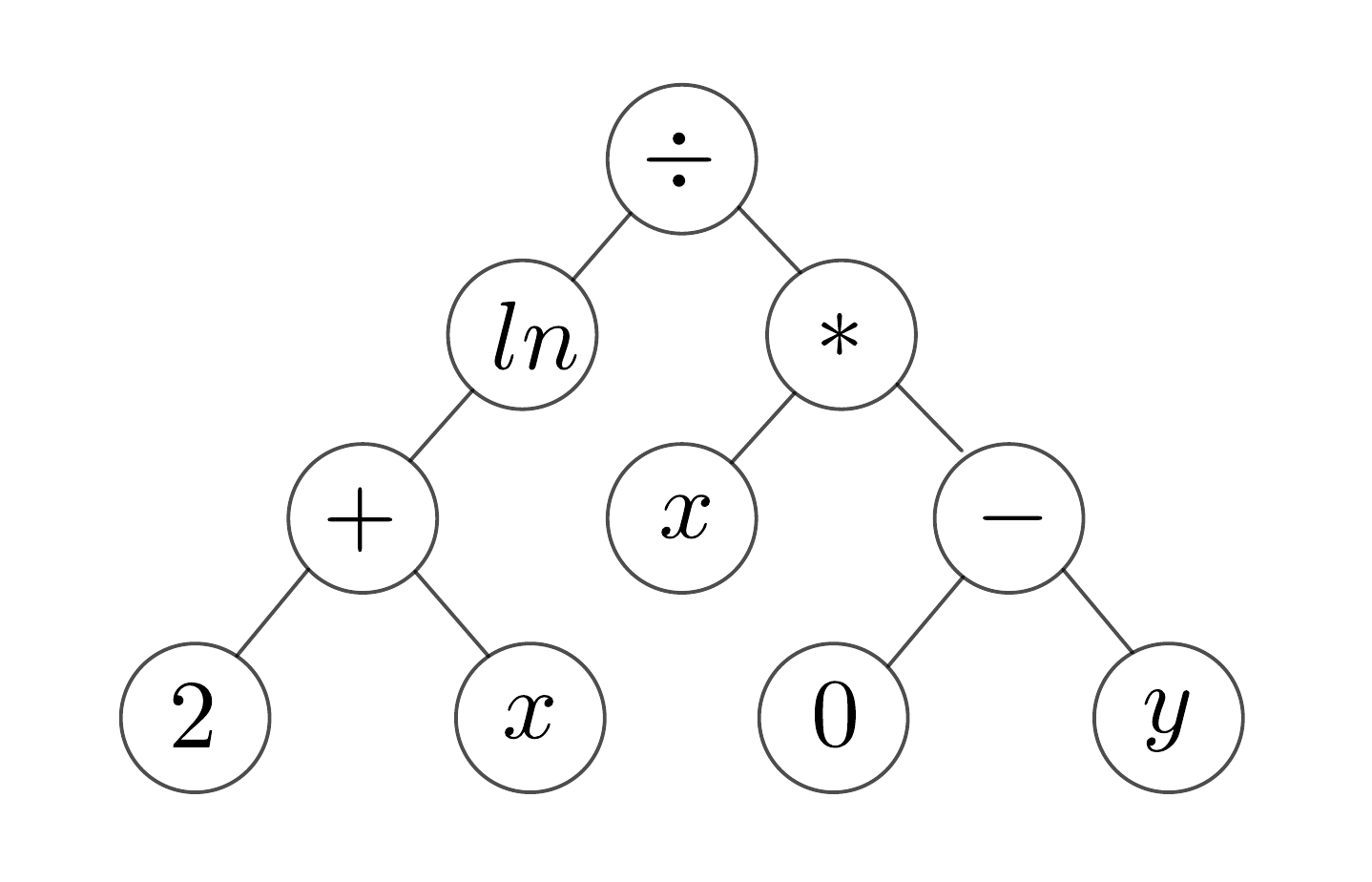}
\caption{A chromosome in tree format representing the function $f(x,y) = \frac{\ln(2+x)}{x(0-y)}$.}
\label{fig:gp-tree}
\end{figure}
	
Mutation is represented by pseudo-randomly replacing a node with a different function or primitive (or one of another set of mutation-like actions), as appropriate. Crossover between chromosomes is enacted by swapping sub-trees. 

For the games in Sections \ref{sec:single} and \ref{sec:twopoint} we examined a number of different sets of hyper-parameters, and the most reasonable for both convergence and computation time were heap sizes up to $10$ on positions with anywhere from one to five heaps.

The package \texttt{gpLearn} is intended for fitting real-valued functions of several real variables to data points using standard elementary functions and binary operations over the reals. We modified the default set of functions to instead focus on discrete functions of several discrete variables. We wrote and included the following binary and unary operations, which operate bitwise on integer inputs: \texttt{XOR}, \texttt{AND}, \texttt{OR}, \texttt{NOT}. We also included \texttt{MOD}, \texttt{LOG$_2$}, and \texttt{PLUS1}, whose operations are self-explanatory. Finally, we introduced logical operators \texttt{EQUAL}, \texttt{LESS}, and \texttt{GREATER} to return $0$ for False and $1$ for True. Default functions included \texttt{SUB} for subtraction, \texttt{ADD}, \texttt{TIMES}, and \texttt{DIVIDE}.  Our fitness function computed the total absolute difference between each genetic program and the computed Grundy values, so that a lower fitness value represents a better fit, although we experimented with measuring distance using the nim-sum. 

Populations ranged from $1{,}000$ to $10{,}000$ individual programs, and we restricted most runs to $20$ generations. \emph{Elites}, relatively highly fit programs in each generation, were retained unmodified between generations. We also experimented with rates of mutation, settling on higher values to prevent getting stuck in local minima. 

\section{A single-point crossover and mutation game}\label{sec:single}
There are two primary methods of crossover used in genetic algorithms, \emph{one-point} and \emph{two-point}. For the former, consider a pair of bit strings of length $n$, $B_1 = (a_1,\ldots ,a_n)$ and $B_2  =(b_1,\ldots ,b_n)$. An integer $k \in [1,n)$ is chosen pseudo-randomly, and the sub-strings $(a_1,\ldots ,a_k)$ and $(b_1,\ldots ,b_k)$ are swapped, leading to the new bit strings $$B_1' = (b_1,\ldots ,b_{k},a_{k+1},\ldots ,a_n), B_2' = (a_1,\ldots ,a_{k},b_{k+1},\ldots ,b_n).$$ After crossover there is a possible mutation, depending on the chosen mutation rate, turning, say, $B_1'$ into $B_1'' = (b_1,\ldots ,b_{i-1},1-b_i,b_{i+1},\ldots ,b_{k},a_{k+1},\ldots ,a_n)$.

Motivated by these processes we define a new impartial combinatorial game, \textsc{ga1}. In order to simplify both rules and analysis we define a position as a single bit string. A mutation move flips a single bit in the string, and while there is no real crossover in a single string we consider the flip of a sequence of bits to be representative of this operation.

\begin{Ruleset}[\textsc{ga1}]
A position in \textsc{ga1} is a bit string of length $n$. There are two move options. Crossover consists of choosing an integer $k$, $1\leq k\leq (n-1)$, wherein all bits from position $1$ through $k$ are flipped. A mutation move is simply the flip of any single bit in the string. A move is legal only if the total number of sub-strings of the form $01$ and $10$ increases.
\end{Ruleset}

This latter restriction, that `disorder' increases, serves two purposes. Firstly, it ensures that the game ends in a finite number of moves. Secondly, it represents the tendency of chromosomes to combine in ever more complex ways over time. We define the condition of increasing sub-strings $01$ and $10$ formally as follows. 

\begin{Definition}
The \emph{entropy} of a bit string game is the number of sub-strings of the form $01$ and $10$. 
\end{Definition}

The game \textsc{ga1} is equivalent to a heap game in the following way. If we consider a \emph{run} in a bit string to be a maximal sub-string consisting of all $0$s or all $1$s, then any bit string can be converted into a list of integers representing run sizes. For example, the string $001000111$ becomes $(2,1,3,3)$. Although this representation loses information about which bits are associated with each integer, the symmetry of the ruleset makes this lost information unnecessary to the game analysis. We can simplify the ruleset further by Proposition~\ref{prop:ga1simplify}.

\begin{Proposition}\label{prop:ga1simplify}
If $H = (h_1,\ldots ,h_k)$ is a list of heaps representing a bit string in \textsc{ga1}, then
\begin{enumerate}
	\item Any heap equal to $1$ can be removed
	\item The order of the heaps does not affect the Grundy value
	\item Each move is equivalent to one of the following
	\begin{enumerate}
		\item Split any heap $h_i > 3$ into two heaps of size at least $2$ each
		\item Remove $1$ from any heap $h_i\geq 3$
		\item Remove $1$ from any heap $h_i\geq 5$ and split the remainder into two heaps of at least $2$ each
		\item Remove any heap of size $2$ or $3$
	\end{enumerate}
\end{enumerate}
\end{Proposition}
\begin{proof}
We will prove each part of Proposition~\ref{prop:ga1simplify} separately.
\begin{enumerate}
	\item A single bit between two runs of the opposite value or at the end of a string is represented by a heap of size $1$. No move that increases entropy has an effect on this heap, and thus its removal has no effect on game play nor the Grundy value of the position. Thus it can be removed.
	
	\item Say that heaps $h_i$ and $h_j$ switch positions in $H$ resulting in $H'$. Any mutation or crossover point $k$ chosen within $h_i$ in $H$ is equivalent to an index $k'$ in $H'$, which results in identical game play. Therefore the order of the heaps in $H$ does not affect the game value. 
	\item For the move equivalences, note that in order to make a legal crossover move in \textsc{ga1} a player must choose the crossover point in the midst of a run. This effectively splits the run into two, and leaves the others alone (other than switching the bits in the affected sub-string). This is equivalent to splitting a heap into two and if one or both of the resulting heaps have size $1$ then they can be removed from play. Similarly, a legal mutation move must also occur in the midst of a run, splitting a heap into either two or three with at least one heap of size $1$. Again, these size $1$ heaps can be removed. 
\end{enumerate}
\end{proof}

As a direct result of Proposition~\ref{prop:ga1simplify} we need only consider single heap positions, since the Grundy value of a list of heaps is equal to the nim-sum of the Grundy values of the individual heaps.

The package \texttt{gpLearn} was employed as described in Section~\ref{sec:methods}. While no exact formula was found, after $14$ generations a local minimum was reached. Modifying hyper-parameters and running for another $7$ generations led to the formula 
\begin{center}
\texttt{MOD(1+h,MOD(h+1,3) + 1) - MOD(h-1,4) + MOD(1+h,3) + 4}
\end{center}

where $h$ is the size of a single heap and \texttt{MOD(x,n)} represents $x \pmod n$. While not a particularly accurate formula we do see the presence of both \emph{modulo 3} and \emph{modulo 4}. Hence, we examine the actual Grundy values closely for periodicity of order twelve and find a striking similarity with the values of the combinatorial game \textsc{kayles}.

\begin{Ruleset}[\textsc{kayles}]\label{rs:kayles}{\rm \cite{dudeneycanterbury}}
In \textsc{kayles} a player may remove one or two stones from any heap, and if any stones remain these may be split into two heaps. 
\end{Ruleset}

\textsc{kayles} has octal code $0.77$ \cite{conway2000numbers} and has been well-studied. In particular, it is known that the Grundy values for a single heap game of \textsc{kayles} of size $n$ is periodic with period $12$ after $n=71$ \cite{guy1956g}. 

\begin{Theorem}\label{thm:ga1-kayles}
The Grundy value of a single heap game of size $n$ in \textsc{ga1} is equal to the value of a heap of size $(n-1)$ in \textsc{kayles}.
\end{Theorem}
\begin{proof}
This is easy to compute for $n\leq 3$. If $n\geq 4$ then the options are $\{(j,n-j) : 2\leq j\leq (n-j)\} \cup \{(k,n-k-1):2\leq k\leq n-k-1\} \cup \{(n-1), (n-2)\}$. The options for an $(n-1)$-sized heap in \textsc{kayles} are $\{(j,n-j-2) : 1\leq j\leq (n-j-2)\} \cup \{(k,n-k-3):1\leq k\leq n-k-3\} \cup \{(n-2),(n-3)\}$. We can therefore consider a move in \textsc{ga1} to be equivalent to the following process:
\begin{enumerate}
	\item Remove a stone from a heap,
	\item Make a \textsc{kayles} move in the resulting heap of size $(n-1)$,
	\item Add a stone back to all resulting heaps.
\end{enumerate}

Therefore the game \textsc{ga1} reduces to a game of \textsc{kayles}, and thus the Grundy values are computable in the same manner as those for \textsc{kayles}.
\end{proof}

\section{A two-point crossover game}\label{sec:twopoint}
Next we consider a similar impartial game based on genetic crossover, this time using two positions instead of one. Consider a pair of bit strings
$$B_1 = (a_1,\ldots ,a_n)$$ and $$B_2  =(b_1,\ldots ,b_n).$$ If $1\leq x < y\leq n$ are integers then \emph{two-point crossover} using positions $x$ and $y$ results in the bit strings $$B_1' = (a_1,\ldots ,a_{x-1},b_x,\ldots ,b_{y-1},a_y,\ldots ,a_n)$$ and  $$B_2' = (b_1,\ldots ,b_{x-1},a_x,\ldots ,a_{y-1},b_y,\ldots ,b_n).$$ That is, a sub-string with matching indices from each bit string is swapped. We wish to define an impartial game motivated by two-point crossover as a move mechanic. As we did with \textsc{ga1}, we play only in a single bit string.
This also means that defining mutation-type moves is redundant since any such move would be equivalent to crossover with $x = y-1$. 

\begin{Ruleset}[\textsc{ga2}]\label{rs:ga2}
A position in \textsc{ga2} is a bit string of length $n$. On their turn a player chooses two integers $x,y\in [1,n], x<y$, wherein all bits from position $x$ through $(y-1)$ are flipped. A move is legal only if the total number of sub-strings of the form $01$ and $10$ increases.
\end{Ruleset}

As with \textsc{ga1} we can reduce \textsc{ga2} to a game on heaps. Note that, again, a run of bits can be represented by an integer. A legal move requires that at least one of $\{x,y\}$ is chosen within a run. The possible options are
\begin{enumerate}
	\item Both $x$ and $y$ are within the bounds of a single run, equivalent to splitting a single heap into three heaps,
	\item $x$ and $y$ are each within the bounds of different runs, equivalent to splitting any two heaps into two each,
	\item $x$ and $y$ are chosen so that exactly one single heap is split into two.
\end{enumerate}

Just as with Proposition~\ref{prop:ga1simplify} we see that heaps of size $1$ are negligible, as is the order of the heaps. However, since players can alter multiple heaps in a single move we cannot compute the Grundy value by simply computing the nim-sum of the Grundy values of single heap games. 

As in Section~\ref{sec:single} we applied \texttt{gpLearn} with the modified function list to computationally determined Grundy values, without first examining these values. Once the number of non-zero heaps was included as a primitive value in games with more than two heaps (e.g. $(3,h_1,h_2,h_3)$ is a three-heap game, while $(4,h_1,h_2,h_3,h_4)$ represents a four-heap position), genetic programming proved much more successful, yielding the formulas below with $100\%$ accuracy:
\begin{enumerate}
	\item For a single-heap game with heap size $h$, 
	\\
	\texttt{MOD(SUB(h,1),PLUS1(PLUS1(1)))}
	\\
	which is equivalent to $(h-1) \pmod{3}$.
	
	\item With two heaps $h_1,h_2$
	\\
	\texttt{MOD(PLUS1(SUB(ADD($h_1$, $h_2$), XOR($h_1$, $h_1$))), \\PLUS1(PLUS1(EQUAL($h_1$, $h_1$))))}
	\\
	which is equivalent to $(h_1+h_2 + 1) \pmod{3}$.
	
	\item For a three-heap game with inputs $3,h_1,h_2,h_3$ we found 
	\\
	\texttt{MOD(ADD(ADD($h_3$, $h_1$), $h_2$), ADD(3, SUB(0, 0)))}
	\\
	which reduces to $(h_1+h_2+h_3) \pmod{3}$.
\end{enumerate}

While these results themselves do not provide a generalized formula, they do generalize easily to the following. 

\begin{Theorem}\label{thm:ga2}
Let $H = (h_1,\ldots ,h_n)$ be an $n$-heap position in \textsc{ga2}, and let $t$ be the smallest non-negative integer such that $(n+t)\equiv 0\pmod{3}$. Then the Grundy value of $H$ is $\left(t + \sum \limits_{i=1}^n h_i\right) \pmod{3}$.
\end{Theorem}
\begin{proof}
Note first that while we can eliminate heaps of size $1$ in our analysis of \textsc{ga2} just as we did in 
Proposition~\ref{prop:ga1simplify} for \textsc{ga1}, we are not compelled to do so. In fact, not removing them makes for a simpler analysis here. 

In the case of a single stone it is clear that the Grundy value is $0$ as no moves are possible. It is also easy to see that the claim holds when all heaps have size $1$ except possibly a single heap of size $2$, so we need only consider the remaining cases. We proceed now by minimum counter-example. Assuming that the claim is false, let $m$ be the smallest integer such that not all games on $m$-many stones follow the statement of the theorem. Among all such games with $m$ stones, let $H = (h_1,\ldots ,h_j)$ be a position with the greatest number of heaps. 

For any positive integer $x$ let $x_1,x_2,x^1,x^2,x^3$ be positive integers such that $x_1+x_2=x$ and $x^1+x^2+x^3=x$. For any $i,k$ with $1\leq i< k\leq j$, the options of $H$ are \\
\noindent
$H\setminus \{h_i\} \cup \{h_{i1},h_{i2}\}$
\\$H\setminus \{h_i\} \cup \{h_i^1,h_i^2,h_i^3\}$
\\$H\setminus \{h_i,h_k\} \cup \{h_{i1},h_{i2},h_{k1},h_{k2}\}$
\\
i.e. all positions in which any one heap, $h_i$, of sufficient size is removed and replaced with two or three heaps whose sum is $h_i$, and those in which any two heaps $h_i, h_k\geq 2$ are removed and each replaced with two heaps whose sums are $h_i,h_k$ respectively.

All options contain $m$ total stones since we have not removed any. Further, every option has more heaps than does $H$, and therefore, by the choice of minimal counter-example, adhere to the statement of the claim. Thus their Grundy values are all equal to $(m+t+1) \pmod{3}$ and $(m+t+2) \pmod{3}$. Note also that both of these values must appear at least once among the Grundy values of the options of $H$. Therefore $H$ must have Grundy value $(m+t) \pmod{3}$, contradicting the claim that $H$ fails the claim of the theorem.
\end{proof}

\section{The Crossover-Mutation Game}

In order to represent both crossover and mutation more accurately, we now consider a game played on a pair of bit strings. 

\begin{Ruleset}[\textsc{crossover-mutation (cm)}]\label{rs:ga3}
A position in \textsc{crossover-mutation} is a pair of bit strings of length $n$, $B_1 = (a_1,\ldots ,a_n)$ and $B_2  =(b_1,\ldots ,b_n)$. There are two move options. \emph{Crossover} consists of choosing an integer $k$, $1\leq k\leq (n-1)$, wherein all bits $1$ through $k$ from $B_1$ are swapped with the bits $1$ through $k$ of $B_2$. In particular, leading to the new bit strings $$B_1' = (b_1,\ldots ,b_{k},a_{k+1},\ldots ,a_n), B_2' = (a_1,\ldots ,a_{k},b_{k+1},\ldots ,b_n).$$ \emph{Mutation} involves choosing a single gene $c_i$ from either of the bit strings and flipping it to $1-c_i$. In both cases, the move is legal if the total number of sub-strings of the form $01$ and $10$ increases.
\end{Ruleset}

All positions of \textsc{crossover-mutation} are equivalent to certain positions from another game called \textsc{arc kayles}. We first present the ruleset, then prove the equivalence.

\begin{Ruleset}[\textsc{arc kayles}] {\rm \cite{Schaefer1978}} Let $G$ be a graph. On a player's turn, they remove an edge of $G$ along with all edges incident to it. 
\end{Ruleset}

\begin{Theorem}
Let $G$ be a \textsc{cm} position. $G$ is equivalent to an \textsc{arc kayles} position.  
\end{Theorem}

\begin{proof}
Let $G_{CM}$ be a \textsc{cm} position of length $n$ as $B_1 = (a_1,\ldots, a_n)$ and $B_2  =(b_1,\ldots, b_n)$. We will first construct the \textsc{arc kayles} position, $G_{AK}$. Then we will prove its equivalence by showing that there is a bijection between the options of the games. 

First consider $B_1$ of $G_{CM}$. For each mutation, its representation in $G_{AK}$ is an edge. Edges are incident in $G_{AK}$ if the corresponding bits in $G_{CM}$ were adjacent in $B_{1}$. Similarly for $B_2$. We label the edges of $G_{AK}$ by the corresponding bit label in $B_1$ or $B_2$ respectively. For the crossover moves in $G_{CM}$, if there exists a crossover move at $a_{i}$, $a_{i+1}$ and $b_{i}$, $b_{i+1}$ then in $G_{AK}$ there is a vertex connecting the edges $a_i$ and $a_{i+1}$, call it $v_{a_i,a_{i+1}}$, similarly for $b_{i}$ and $b_{i+1}$, call it $v_{b_i,b_{i+1}}$. Label this edge as $v_{a_i,a_{i+1}}v_{b_i,b_{i+1}}$. (see Figure~\ref{fig:example} for an example of the equivalence).

To show that $G_{AK}$ is equivalent to $G_{CM}$ via this construction, we need to show that there exists a bijection between the options. In particular, that $G_{CM} - G_{AK}=0$. Since the rulesets are impartial, we consider $G_{CM} + G_{AK}$. Suppose the first player moves in $G_{CM}$ with a mutation at $a_{i}$. By the existence of this mutation, it means that both $a_{i-1}$ and $a_{i+1}$ were the same as $a_{i}$ (if they exist), otherwise the entropy wouldn't have increased. After the turn, neither can be mutated thereafter because again, it would not increase the entropy. Also, this move disallows future crossover at $a_{i}$ because it will not increase the entropy. Player 2 responds by removing the edge $a_{i} \in G_{AK}$. This has the effect of removing all incident edges, in particular, $a_{i-1}$, $a_{i+1}$ and $v_{a_i,a_{i+1}}v_{b_i,b_{i+1}}$, if they exist. If instead Player 1 chose a crossover move in $G_{CM}$ at position $k$, this eliminates the possibility of future mutations at positions $a_{k}$, $a_{k-1}$, $b_{k}$, and $b_{k-1}$. The corresponding move for Player 2 is to respond in $G_{AK}$ by removing the edge with label $v_{a_{k-1},a_{k}}v_{b_{k-1},b_{k}}$, which effectively removes all edges $a_{k}$, $a_{k-1}$, $b_{k}$, and $b_{k-1}$.

 If instead Player 1 moved in $G_{AK}$, we simply reverse the roles in the above argument and Player 2 will always have a response. Thus Player 2 will win this game under normal play. Hence $G_{CM}$ and $G_{AK}$ are equivalent. 
\end{proof}

It turns out that \textsc{cm} is also closely related to another well-studied game.

\begin{Ruleset}[\textsc{cram}]\label{rs:cram}{\rm\cite{Berlekamp}}
In the impartial game \textsc{cram} players take turns filling a pair of empty orthogonally adjacent spaces in a grid.
\end{Ruleset}

The reader may recognize \textsc{cram} as the impartial version of \textsc{domineering}. All \textsc{cm} positions are also associated with $2\times n$ \textsc{cram} positions, except for a few with extra pendant vertices which, if realized in \textsc{cram}, require a board of width at least three. We address one such case below. If the \textsc{cm} position is of a certain form, in particular every entry $a_i$ of $B_1$ is the same, and every entry of $B_2$ is $1-a_i$, the proven equivalence to a subset of \textsc{arc kayles} positions allows us to immediately deduce the game values. 

\begin{Theorem}[\cite{Berlekamp}, vol 3]\label{thm:ww-grids}
Let $G$ be a position in \textsc{arc kayles} in the form of a $2\times n$ grid graph. Then $G$ has value $0$ if $n$ even and value $*$ if $n$ odd. Furthermore, this game value does not change under the addition of up to two \emph{tufts} (i.e. induced stars whose center is a vertex of the grid graph).
\end{Theorem}

\begin{Theorem}\label{claim:ak}
Let $G(k)$ be a position in \textsc{arc kayles} in the form of a $2\times k$ grid graph with pendant edges adjacent to $3$ or $4$ of the four corners (see Figure \ref{fig:grid}). Then $G(2k+1)$ has game value $*2$ if $k\in \{0,1\}$ and $*$ if $k\geq 2$, and $G(2k)$ has value $0$ for all $k\geq 1$ when $h_0$ is present. 
\end{Theorem}

\begin{figure}
\centering
\includegraphics[width=10cm]{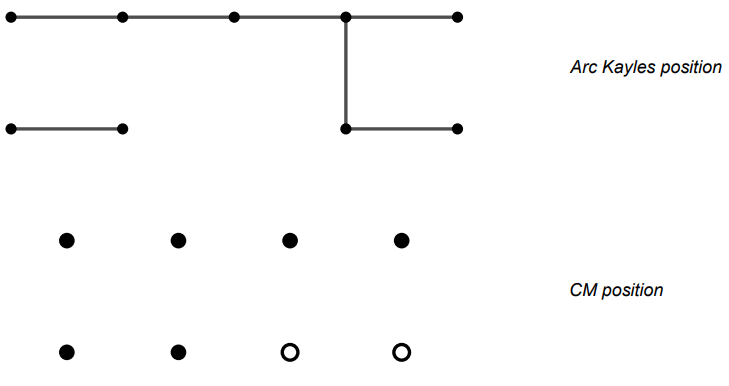}
\caption{Example of an \textsc{arc kayles} position which is equivalent to a position in \textsc{crossover-mutation}.}
\label{fig:example}
\end{figure}

\begin{figure}
\centering
\includegraphics[width=\textwidth]{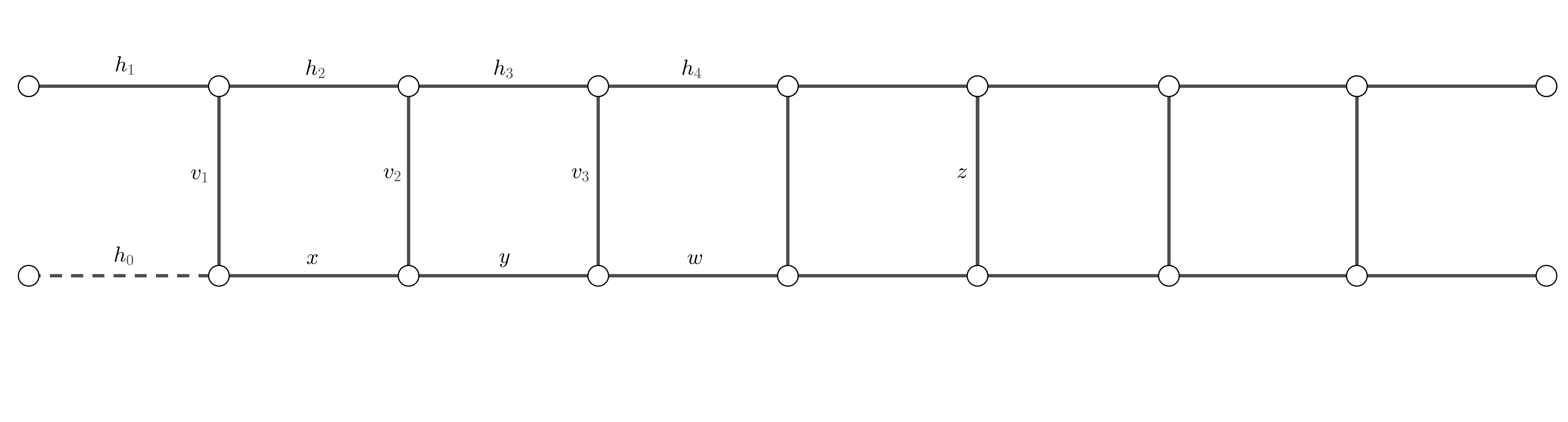} 
\caption{An \textsc{arc kayles} position; equivalent to a position in \textsc{crossover-mutation} when $h_0$ is present.}
\label{fig:grid}
\end{figure}

\begin{figure}
\centering
\includegraphics[width=\textwidth]{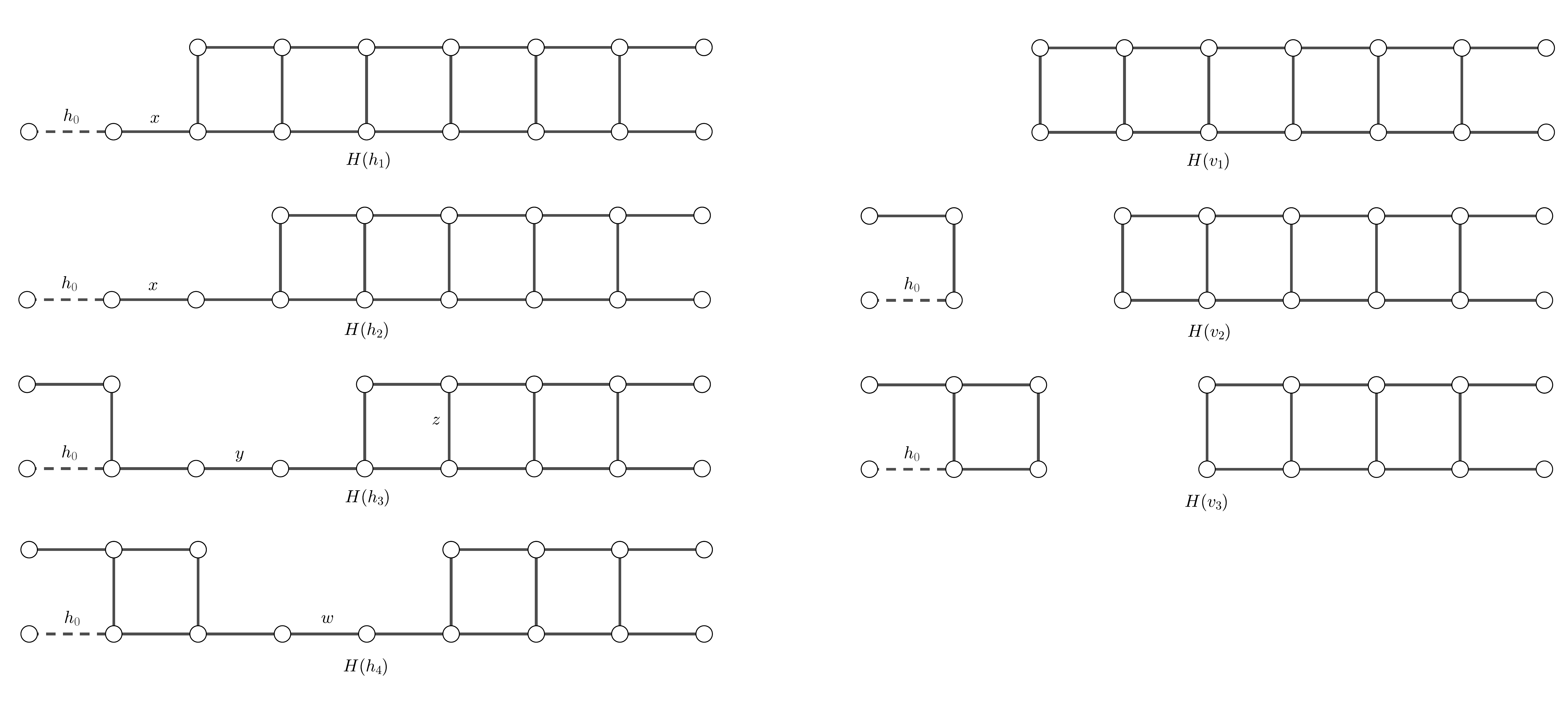}
\caption{The options of $G(2k+1)$ from Figure \ref{fig:grid}.}
\label{fig:grid-options}
\end{figure}

\begin{proof}
Note that if $k\leq 1$ then the possible values of $G(2k+1)$ are easily demonstrated by exhaustion. The value of $G(2k)$ is just as easily found to be in $\mathcal{P}$ by considering an involution strategy, whereby the second player responds to a play on edge $e$ with a play on the edge equivalent to $e$ under $180^{\circ}$ rotational symmetry. We now proceed by induction on $k$ to find the remaining values of $G(2k+1)$ whether or not edge $h_0$ is present.

Let $e$ be an edge in $G(2k+1)$, and consider $H(e)$ to be the option yielded by play on $e$ (see Figure \ref{fig:grid-options}). We demonstrate that no option of $G(2k+1)$ has value $*$. 

\begin{itemize}
	\item[$H(h_1)$] Play on edge $x$ results in a graph of the form $2\times (2k-1)$ with three pendant edges. If $k$ is sufficiently large this graph has value $*$ by inductive assumption, and hence $H(h_1)$ does not have value $*$. Otherwise, the value can be checked exhaustively for the base case of $G(5)$, when $k=2$, to have value $*$ with or without the presence of $h_0$. Hence, $H(h_1)$ does not have value $*$. 
	
	\item[$H(h_2)$] Play on the edge $x$ results in a position with value $*$ by Theorem \ref{thm:ww-grids}. Therefore $H(h_2)$ does not have value $*$.
	
	\item[$H(h_3)$] If $h_0$ is not present then play on edge $y$ yields a path with value $*$ disconnected from a $2\times (2k-2)$ grid graph with two pendant edges which, by Theorem \ref{thm:ww-grids}, has value $0$. If $h_0$ is present then play on edge $z$ yields the sum of a small graph with value $*$ and a $2\times(2k-4)$ grid graph with two pendant edges. In both cases, the resulting sums are $*$. Therefore, $H(h_3)$ does not have value $*$.
	
	\item[$H(h_4)$] Here $h_4$ can be any horizontal edge to the right of $h_3$. Play on edge $w$ results in a game with a sum of two positions with opposite parity. Hence has value $* + 0 = *$ by Theorem \ref{thm:ww-grids}, so $H(h_4)$ does not have value $*$.
	
	\item[$H(v_1)$] This graph has value $0$ by Theorem \ref{thm:ww-grids}.
	
	\item[$H(v_2)$] If $h_0$ is present then we have the sum of a path with value $*2$ and a game with value $*$ by Theorem \ref{thm:ww-grids}. If $h_0$ is not present then the path has value $*$. So $H(v_2)$ has value $*3$ or $0$.
	
	\item[$H(v_3)$] We invoke Theorem \ref{thm:ww-grids} yet again, as the resulting graph is a pair of grid graphs with one or two pendant edges each, both with value $*$ or both with value $0$. Therefore $H(v_3)$ has value $0$.

\end{itemize}

Since no option of $G(2k+1)$ has value $*$ and $G(2k+1)\in \mathcal{N}$, we see that it has value $*$ for $k\geq 2$.
\end{proof}

Theorem \ref{claim:ak} leads directly to the following corollary about a family of \textsc{crossover-mutation} positions.

\begin{Corollary}\label{cor:cm-extreme-position}
The \textsc{cm} game composed of a length-$n$ string of all $1$s and a length-$n$ string of all $0$s has value $0$ if $n$ is odd, $*2$ if  $n\in \{2,4\}$, and $*$ otherwise.
\end{Corollary}
\begin{proof}
This position is equivalent to the \textsc{arc kayles} position $G(n-1)$ with $h_0$ present, as indicated in Theorem \ref{claim:ak}. 
\end{proof}

Most remaining \textsc{cm} positions are equivalent to $2\times n$ positions in \textsc{cram} which, while remaining unsolved, have been addressed in the literature~\cite{Berlekamp}. It's worth noting that all \textsc{cm} positions in which no crossover move is possible are simply represented by a disjunctive sum of paths in \textsc{arc kayles}, whose values are known~\cite{HugganS}.

\section{Conclusion and further research}\label{sec:conc}
We have seen the possible application of genetic programming to the determination of Grundy values of impartial combinatorial games. In addition, we have seen it both provide an exact function and simply inform our own mathematical analysis. Note that the game for which it proved most useful, \textsc{ga2}, could likely have been solved without the use of genetic programming and instead through a simple examination of the computed Grundy values. But we have also seen that it \emph{was} solved through the use of genetic programming, and therefore this method could prove useful in the future. At the very least, it could be utilized to reduce the time and effort taken to conjecture formulas for Grundy values. 

We are curious whether or not genetic programming can be used for problems within CGT that a mathematician simply examining a list of values is unlikely to solve. To answer this we suggest more efforts into this practice. It will be very useful, for example, to compile a database of impartial combinatorial games with known and as yet unknown solutions. This could help inform the choice of default functions to include in future genetic programming attempts. 

There are modifications that we suggest be made to future GP for CGT projects. Firstly, it would be beneficial to develop a more robust fitness function. As there is no obvious metric over the set of nimbers outside of the nim-sum, an analytical approach to metrics over impartial games would be helpful. Secondly, the method for fitness employed in \cite{oltean2004evolving} does not use pre-computed data points at all. Instead the author determines the fitness of a program by comparing the computed outcome classes of a set of positions with those of its options, and relating the fitness to the number of deviations from the basic tenets of impartial games that are found among these computations. Something similar could be used for Grundy value programming, involving the \emph{mex} (minimum excludant) function. However, the distance between actual value and computed value remains a possible stumbling block.

\end{document}